\documentclass[a4paper,12pt]{article}
\usepackage[top=2cm, bottom=2cm, outer=2.5cm, inner=2cm, headsep=14pt]{geometry}
\usepackage[dvips]{epsfig}
\usepackage{graphicx,hyperref}
\usepackage{amsmath,amssymb,euscript,graphicx,amsfonts,verbatim}
\usepackage{color}
\usepackage{enumitem}
\usepackage{tikz}

\usepackage[utf8]{inputenc}

\definecolor{ForestGreen}{RGB}{12, 110, 46}
\definecolor{ForestGreenTwo}{RGB}{120, 110, 86}

\usepackage{pgfplots}
\pgfplotsset{compat=1.15}
\usepackage{mathrsfs}
\usetikzlibrary{arrows}

\newtheorem{theorem}{Theorem}[section]
\newtheorem{proposition}[theorem]{Proposition}
\newtheorem{lemma}[theorem]{Lemma}
\newtheorem{corollary}[theorem]{Corollary}
\newtheorem{definition}[theorem]{Definition}
\newtheorem{problem}[theorem]{Problem}
\newtheorem{conjecture}[theorem]{Conjecture}
\newtheorem{question}[theorem]{Question}
\newtheorem{remark}[theorem]{Remark}
\newtheorem{example}[theorem]{Example}
\newtheorem{notation}[theorem]{Notation}

\newcommand{\G}{\Gamma}
\newcommand{\Aut}{\mathrm{Aut}}

\newcommand{\Qed}{\rule{2.5mm}{3mm}}
\newenvironment{proof}{{\noindent \sc Proof.}}{\hfill $\Qed$}
\def\ZZ{{\hbox{\sf Z\kern-.43emZ}}}

\title{On some problems regarding distance-balanced graphs}

\author{
{Blas Fern\'andez}\\
{\small Andrej Maru\v si\v c Institute}\\
{\small University of Primorska}\\
{\small Muzejski trg 2, 6000 Koper, Slovenia }\\
{\small blas.fernandez@famnit.upr.si} \and
{Ademir Hujdurovi\'c}\\
{\small Faculty of Mathematics, Natural Sciences and Information Technologies  }\\
{\small University of Primorska}\\
{\small Glagolja\v ska 8, 6000 Koper, Slovenia }\\
{\small ademir.hujdurovic@upr.si}
}

\begin{document}

	\definecolor{uuuuuu}{rgb}{0.26666666666666666,0.26666666666666666,0.26666666666666666}
	\definecolor{ffffff}{rgb}{1,1,1}
	\definecolor{uuuuuu}{rgb}{0.2,0,0}
	\definecolor{uuuuuu}{rgb}{0.26666666666666666,0.26666666666666666,0.26666666666666666}
	\definecolor{uuuuuu}{rgb}{0.2,0,0}
	\definecolor{cqcqcq}{rgb}{0.7529411764705882,0.7529411764705882,0.7529411764705882}
{\small
\maketitle
}

\begin{abstract}
A graph $\Gamma$ is said to be {\em distance-balanced} if for any edge $uv$ of $\Gamma$, the number of vertices closer to $u$ than to $v$ is equal to the number of vertices closer to $v$ than to $u$, and it is called {\em nicely distance-balanced} if in addition this number is independent of the chosen edge $uv$. A graph $\Gamma$ is said to be {\em strongly distance-balanced} if for any edge $uv$ of $\Gamma$ and any integer $k$, the number of vertices at distance $k$ from $u$ and at distance $k+1$ from $v$ is equal to the number of vertices at distance $k+1$ from $u$ and at distance $k$ from $v$. 

In this paper we {solve} an open problem posed by Kutnar and Miklavi\v c \cite{KM} by constructing several infinite families of nonbipartite nicely distance-balanced graphs which are not strongly distance-balanced. We disprove a conjecture regarding characterization of strongly distance-balanced graphs posed by Balakrishnan  {et al.}\  \cite{BCPSSS} by providing infinitely many counterexamples, and  {answer a question} posed by Kutnar  {et al.}\ in \cite{{KMMM05}} regarding {the} existence of semisymmetric distance-balanced graphs which are not strongly distance-balanced by providing an infinite family of such examples. { We also show that for a graph $\G$ with $n$ vertices and $m$ edges it can be checked in $O(mn)$ time if $\G$ is strongly-distance balanced and if $\G$ is nicely distance-balanced.}
\end{abstract}




\section{Introduction}
\label{sec:intro}

\noindent
Let $\G$ be a finite, undirected, connected graph and let  $V(\G)$ and $E(\G)$ denote the vertex set and the edge set of $\G$, respectively.
For $u,v \in V(\G)$, let $d(u,v) = d_{\G}(u,v)$ denote the minimal path-length distance between $u$ and $v$. 
For a pair of adjacent vertices $u,v$ of $\G$ we denote 
$$
W_{u,v} = \{x\in V(\G)\mid d(x,u)<d(x,v)\}.
$$
We say that $\G$ is {\em distance--balanced} (DB for short) whenever
for an arbitrary pair of adjacent vertices $u$ and $v$ of $\G$ we have that
$$
|W_{u,v}| = |W_{v,u}|.
$$

The investigation of distance-balanced graphs was initiated in 1999 by Handa \cite{Ha}, who considered distance-balanced partial cubes. The term itself was introduced by Jerebic, Klav\v zar and Rall in \cite{JKR}, who gave some basic properties and characterized { Cartesian} and lexicographic products of distance-balanced graphs. The family of distance-balanced graphs is very rich and its study is interesting from various purely graph-theoretic aspects where one focuses on particular properties of such graphs such as symmetry~\cite{KMMM05, KMMM09,YHLZ}, connectivity~\cite{Ha, MS} or complexity aspects of algorithms related to such graphs~\cite{CL}. However, the balancedness property of these graphs makes them very appealing also in areas such as mathematical chemistry and communication networks. For instance, the investigation of such graphs is highly related to the well-studied Wiener index and Szeged index (see ~\cite{BBCKVZ14, IKM, JKR, TRA}) and they present very desirable models in various real-life situations related to (communication) networks~\cite{BBCKVZ14}. 
Recently, the relations between distance-balanced graphs and traveling salesman problem were studied in \cite{CD}. It turns out that these graphs can be characterized by properties that at first glance do not seem to have much in common with the original definition from~\cite{JKR}.
 For example, in~\cite{BCPSSS} it was shown that the distance-balanced graphs coincide with the {\em self-median} graphs, that is graphs for which the sum of the distances from a given vertex to all other vertices is independent of the chosen vertex. Other such examples are  {\em equal opportunity graphs} (see~\cite{BBCKVZ14} for the definition). In~\cite{BBCKVZ14} it is shown that { distance-balanced graphs of even order} are also equal to opportunity graphs. Finally, let us also mention that various generalisations of the distance-balanced property were defined and studied in the literature, see for example \cite{AMHK, FrM, HA, JKRus, MS1}.

\medskip
The notion of nicely distance-balanced graphs appears quite naturally in the context of DB graphs. We say that $\G$ is {\em nicely distance--balanced} (NDB for short) whenever there exists a positive integer $\gamma=\gamma(\G)$,
such that for an arbitrary pair of adjacent vertices $u$ and $v$ of $\G$
$$
|W_{u,v}| = |W_{v,u}| = \gamma
$$
holds. Clearly, every NDB graph is also DB, but the opposite is not necessarily true. For example, if $n \ge 3$ is an odd positive integer, then the prism graph on $2n$ vertices is DB, but not NDB. 

Assume now that $\G$ is NDB. Let us denote the diameter of $\G$ by $d$ ({\em diameter} of a graph is the maximum distance between two vertices).  In \cite{KM}, where these graphs were first defined, it was proved that $d \le \gamma$ and NDB graphs with $d=\gamma$ were classified. It turns out that $\G$ is NDB with $d=\gamma$  if and only if $\G$ is either isomorphic to a complete graph on $n \ge 2$ vertices,  or to a cycle on $2d$ or $2d+1$ vertices.  In \cite{FMP} regular NDB graphs for which $\gamma=d+1$ were studied. The situation in this case is much more complex than in the case $\gamma=d$. It was shown that the only regular NDB {graphs} with valency $k$, diameter $d$ and $\gamma = d+1$ are the Petersen graph (with $k=3$ and $d=2$), the complement of the Petersen graph (with $k=6$ and $d=2$), the complete multipartite graph $K_{t \times 3}$ with $t$ parts of cardinality $3$, $t \ge 2$ (with $k=3(t-1)$ and $d=2$), the M\"obius ladder graph on {8} vertices (with $k=3$ and $d=2$);  the Paley graph on 9 vertices (with $k=4$ and $d=2$),  the $3$-dimensional hypercube $Q_3$ (with $k=3$ and $d=3$), the line graph of the $3$-dimensional hypercube $Q_3$ (with $k=4$ and $d=3$), and the icosahedron (with $k=5$ and $d=3$). 

Another concept closely related to the concept of distance-balanced graphs is the one of strongly distance-balanced graphs. For an arbitrary edge $uv$ of a given graph $\G$, and any two nonnegative integers $i,j$ we let
$$D^i_j(u, v)=\{x\in V(\G)\mid d(u, x)=i \textrm{ and } d(v, x)=j\}.$$
A graph $\G$ is called {\it strongly distance-balanced} {(SDB for short)} if $|D^i_{i-1}(u, v)|=|D^{i-1}_i(u, v)|$ holds for every $i\geq 1$ and every edge $uv$ in $\G$. It is easy to see that a strongly distance-balanced graph is also distance-balanced, but the converse is not true in general (see \cite{KMMM05}).  
For more results on this and related concepts see \cite{BCPSSS,CL,IKM,KM,MS}.

In this paper we {solve an open} problem posed by Kutnar and Miklavi\v c \cite{KM} regarding {the} existence of nonbipartite NDB graphs which are not SDB. We construct several infinite families of such graphs, see Proposition \ref{ex:reg} and Corollary~\ref{regular} for {a} construction of regular examples, and Proposition~\ref{mainprop} for {a} construction of non-regular examples. 
In section~\ref{sec:Counterexample to conjecture Sparl} we provide an infinite family of counterexamples to a conjecture {regarding the characterization} of SDB graphs posed by Balakrishnan  {et al.}\  \cite{BCPSSS}. In section~\ref{sec:semisymmetric} we {answer a} question posed by Kutnar {et al.}\ in \cite{{KMMM05}} regarding {the} existence of semisymmetric distance-balanced graphs which are not strongly distance-balanced and provide an infinite family of such examples. In section~\ref{sec:recognition of SDB and NDB} we show that for a graph $\G$ with $n$ vertices and $m$ edges it can be checked in $O(mn)$ time if $\G$ is strongly-distance balanced and { if $\G$ is nicely distance-balanced.}


\section{Preliminaries}
\label{sec:prelim}

\noindent
In this section we recall some preliminary results that we will find useful later in the paper. Let $\G$ denote a simple, finite, connected graph with vertex set $V(\G)$, edge set $E(\G)$. If $u,v \in V(\G)$ are adjacent then we simply write $u \sim v$ and we denote the corresponding edge by $uv=vu$. For $u\in V(\G)$ and an integer $i$ we let $S_{i}(u)$ denote the set of vertices of $V(\G)$ that are at distance $i$ from $u$. We abbreviate $S(u)=S_1(u)$. We set $\epsilon(u)=\max\{d(u,z) \mid z \in V(\G)\}$ and we call $\epsilon(u)$ the {\em eccentricity} of  $u$. Let  $d = \max \{\epsilon(u) \mid u \in V(\G)\}$ denote the {\em diameter} of $\G$. Pick adjacent vertices $u,v$ of $\G$. For any two non-negative integers $i,j$ we let 
$$
D^i_j(u,v)=S_i(u)\cap S_j(v).
$$
By the triangle inequality we observe only the sets $D^{i-1}_i(u,v)$, $D^{i}_i(u,v)$ and $D^{i}_{i-1}(u,v)$ ($1\le i \le d$) can be nonempty (see also Figure~\ref{01}).

{\small\begin{figure}[!ht]{\rm
\begin{center}
\begin{tikzpicture}[scale=.565]

\draw [line width=1pt, draw=ForestGreen] (-2.,-3.)-- (22.,-3.);
\draw [line width=1pt, draw=ForestGreen] (-2.,3.)-- (22.,3.);
\draw [line width=1pt, draw=ForestGreen] (1.,0.)-- (25,0);
\draw [line width=1pt, draw=ForestGreen] (-2,-3)-- (4,3);
\draw [line width=1pt, draw=ForestGreen] (4,-3)-- (10,3);
\draw [line width=1pt, draw=ForestGreen] (10,-3)-- (16,3);
\draw [line width=1pt, draw=ForestGreen] (16,-3)-- (22,3);
\draw [line width=1pt, draw=ForestGreen] (22,-3)-- (25,0);
\draw [line width=1pt, draw=ForestGreen] (-2.,3.)-- (4,-3);
\draw [line width=1pt, draw=ForestGreen] (4,3.)-- (10,-3);
\draw [line width=1pt, draw=ForestGreen] (10,3.)-- (16,-3);
\draw [line width=1pt, draw=ForestGreen] (16,3.)-- (22,-3);
\draw [line width=1pt, draw=ForestGreen] (22,3.)-- (25,0);
\draw [line width=1pt, draw=ForestGreen] (-2,-3)-- (-2,3);
\draw [line width=1pt, draw=ForestGreen] (4,-3)-- (4,3);
\draw [line width=1pt, draw=ForestGreen] (10,-3)-- (10,3);
\draw [line width=1pt, draw=ForestGreen] (16,-3)-- (16,3);
\draw [line width=1pt, draw=ForestGreen] (22,-3)-- (22,3);

\fill (-2.,-3.) circle [radius=0.23];
\fill (-2.,3.) circle [radius=0.23];
\node at (-2.5,3.1) {\normalsize $u$};
\node at (-2.5,-3.1) {\normalsize $v$};

\draw[fill=white, draw=white, line width=0.6pt] (4.,3.) ellipse (1.5cm and .9cm);
\draw[fill=white, draw=white, line width=0.6pt] (7.,0.) ellipse (1.5cm and .9cm);
\draw[fill=white, draw=black, line width=0.6pt] (1.,0.) ellipse (1.5cm and .9cm);
\draw[fill=white, draw=white, line width=0.6pt] (4.,-3.) ellipse (1.5cm and .9cm);
\draw[fill=white, draw=white, line width=0.6pt] (19.,0.) ellipse (1.5cm and .9cm);
\draw[fill=white, draw=black, line width=0.6pt] (22.,3.) ellipse (1.5cm and .9cm);
\draw[fill=white, draw=black, line width=0.6pt] (22.,-3.) ellipse (1.5cm and .9cm);
\draw[fill=white, draw=black, line width=0.6pt] (10.,3.) ellipse (1.5cm and .9cm);
\draw[fill=white, draw=black, line width=0.6pt] (10.,-3.) ellipse (1.5cm and .9cm);
\draw[fill=white, draw=black, line width=0.6pt] (13.,0.) ellipse (1.5cm and .9cm);
\draw[fill=white, draw=white, line width=0.6pt] (16.,3.) ellipse (1.5cm and .9cm);
\draw[fill=white, draw=white, line width=0.6pt] (16.,-3.) ellipse (1.5cm and .9cm);
\draw[fill=white, draw=black, line width=0.6pt] (25.,0.) ellipse (1.5cm and .9cm);

\node at (1,0) {\normalsize $D^{1}_{1}$};
\node at (7,0) {\normalsize $\cdots$};
\node at (13,0) {\normalsize $D^{i}_{i}$};
\node at (19,0) {\normalsize  $\cdots$};
\node at (25,0) {\normalsize $D^{d}_{d}$};

\node at (4,3) {\normalsize $\cdots$};
\node at (10,3) {\normalsize $D^{i-1}_{i}$};
\node at (16,3) {\normalsize  $\cdots$};
\node at (22,3) {\normalsize $D^{d-1}_{d}$};

\node at (4,-3) {\normalsize $\cdots$};
\node at (10,-3) {\normalsize $D^{i}_{i-1}$};
\node at (16,-3) {\normalsize  $\cdots$};
\node at (22,-3) {\normalsize $D^{d}_{d-1}$};
\end{tikzpicture}
\caption{\rm 
Graphical representation of the sets $D^i_j(u,v)$. The line between $D_j^i$ and $D_{m}^{\ell}$ indicates possible edges between vertices of $D^i_j$ and $D_{m}^{\ell}$.
}
\label{01}
\end{center}
}\end{figure}
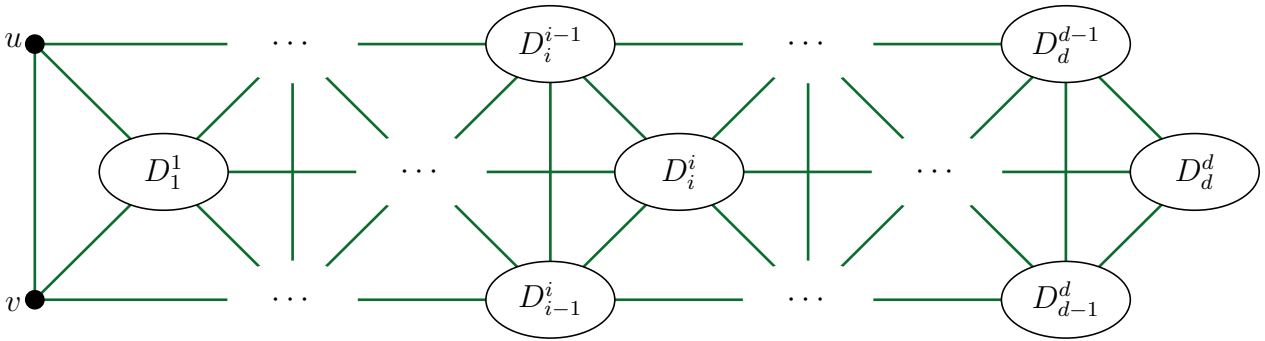}

Let us recall the definition of {NDB graphs}. For an edge $uv$ of $\G$ we denote 
$$
W_{u,v} = \{x\in V(\G)\mid d(x,u)<d(x,v)\}.
$$
We say that $\G$ is {\em nicely distance--balanced} (NDB for short) whenever there exists a positive integer $\gamma=\gamma(\G)$, such that for any edge $uv$ of $\G$
$$
|W_{u,v}| = |W_{v,u}| = \gamma
$$
holds. One can easily see that $\G$ is NDB if and only if for every edge $uv\in E(\G)$ we have

\begin{equation}
	\label{er}
	\sum_{i=1}^{d}|D^{i}_{i-1}(u,v)|=\sum_{i=1}^{d}|D^{i-1}_i(u,v)|=\gamma. \nonumber 
\end{equation}
Pick adjacent vertices $u,v$ of $\G$. For the purposes of this paper we say that the edge $uv$ is {\em balanced}, if $|W_{u,v}| = |W_{v,u}|$ holds.

Another concept closely related to the concept of distance-balanced graphs is the one of strongly distance-balanced graphs. A graph $\G$ is called {\it strongly distance-balanced} {(SDB for short)} if $|D^i_{i-1}(u, v)|=|D^{i-1}_i(u, v)|$ holds for every $i\geq 1$ and every edge $uv$ in $\G$. { Please note SDB graphs are also called {\it distance-degree regular} and were first studied in \cite{HN84}.} It is easy to see that a strongly distance-balanced graph is also distance-balanced, but the converse is not true in general (see \cite{KMMM05}).

Kutnar  {et al.}\ gave the following characterization of strongly distance-balanced graphs.
\begin{proposition}\cite[Proposition 2.1]{KMMM05}\label{prop:SDB characterization}
	Let $\G$ be a graph with diameter $d$. Then $\G$ is strongly distance-balanced if and only if $|S_i(u)|=|S_i(v)|$ holds for every edge $uv\in E(\G)$ and every $i\in \{0,\ldots,d\}$.
\end{proposition}

We say that an edge $uv$ of a graph $\G$ is {\em strongly distance-balanced} if   $|D^i_{i-1}(u, v)|=|D^{i-1}_i(u, v)|$ holds for every $i\geq 1$. From the proof of \cite[Proposition 2.1]{KMMM05} the following result can be obtained. We include the proof here for the sake of completeness.

\begin{lemma}\label{lem:edges strongly distance-balanced}
	Let $\G$ be a graph with diameter $d$, and $uv$ an arbitrary edge of $\G$. Then the edge $uv$ is strongly distance-balanced if and only if $|S_i(u)|=|S_i(v)|$ for every $i\in \{1,\ldots,d\}$.
\end{lemma}

\begin{proof}
Assume first the edge $uv$ of $\G$ is strongly distance-balanced. Then, by definition, we have   $|D^i_{i-1}(u, v)|=|D^{i-1}_i(u, v)|$ for every $i\geq 1$. However, since $S_i(u)= D^i_{i+1}(u,v) \cup D^i_{i}(u,v) \cup D^i_{i-1}(u,v) $ (disjoint union) and $S_i(v)= D^{i-1}_{i}(u,v) \cup D^i_{i}(u,v) \cup D^{i+1}_i(u,v) $ (disjoint union), we have also $|S_i(u)|=|S_i(v)|$ for every $i\in \{1,\ldots,d\}$. 

Next assume  that $|S_i(u)|=|S_i(v)|$ holds for every $i\in \{1,\ldots,d\}$. Using induction we show that $|D^i_{i-1}(u, v)|=|D^{i-1}_i(u, v)|$ holds for every $i\in \{1,\ldots,d\}$. Obviously, we have $|D^1_{0}(u, v)|=|D^{0}_1(u, v)|=1$. Suppose now that $|D^k_{k-1}(u, v)|=|D^{k-1}_k(u, v)|$ holds for $1\leq k \leq d$. We observe
\begin{eqnarray}
	|D^k_{k+1}(u,v)|&=&|S_k(u)|-|D^k_k(u,v)|-|D^k_{k-1}(u,v)| \label{dkk1} \\ 
	|D^{k+1}_k(u,v)|&=&|S_k(v)|-|D^k_k(u,v)|-|D^{k-1}_k(u,v)|  \label{dkk2}
\end{eqnarray} 
Since $|S_k(u)|=|S_k(v)|$ and in view of the induction hypothesis, $|D^k_{k-1}(u,v)|=|D^{k-1}_k(u,v)|$, it follows from \eqref{dkk1} and \eqref{dkk2} that $|D^k_{k+1}(u,v)| = 	|D^{k+1}_k(u,v)|$. This finishes the proof. 
\end{proof}

\bigskip 
{\em {An} Automorphism} of a graph is a permutation of its vertex set that preserves the adjacency relation of the graph. The set of all automorphisms of a graph $\G$ is called {\em the automorphism group} and denoted by $\Aut(\G)$. A graph is {\em vertex-transitive} if its automorphism group acts transitively on the vertex-set, and it is called {\em edge-transitive} if its automorphism group acts transitively on the edge set.
Kutnar  {et al.}\ \cite{KMMM05} used Proposition~\ref{prop:SDB characterization} to prove that vertex-transitive graphs are strongly distance-balanced.
Lemma~\ref{lem:edges strongly distance-balanced} implies that in order to check if a given graph is strongly distance-balanced, one only needs to check the pairs of adjacent vertices that belong to different orbits under the action of the automorphism group of the graph.


\section{Constructions of nonbipartite NDB graphs that are not SDB}

Nicely distance-balanced graphs were studied in \cite{KM}, 
where it is proved that in the class of bipartite graphs, DB and NDB properties coincide, while there are examples of bipartite {NDB} graphs that are not {SDB} given by Handa \cite{Ha}. In \cite{KM} examples of nonbipartite SDB graphs that are not NDB were constructed and the following problem was posed.
\begin{problem}\cite[Problem 3.3]{KM}\label{problem:NDB not SDB}
Find a nonbipartite NDB graph which is not SDB.
\end{problem}

In this section we will construct several infinite families of nonbipartite NDB graphs which are not SDB and so, solve Problem~\ref{problem:NDB not SDB}. To do this, we first study the { Cartesian} product of graphs. 
NDB graphs in the framework of the { Cartesian} graph product were studied in \cite{KM}. We start this section with the definition of this product. 

Let $G$ and $H$ denote connected graphs. The {\it{{ Cartesian} product of $G$ and $H$}}, denoted by $G \square H$, is the graph with vertex set $V(G)\times V(H)$ where two vertices $(g_1, h_1)$ and $(g_2, h_2)$ are adjacent if and only if $g_1=g_2$ and $h_1 \sim h_2$ in $H$, or $h_1=h_2$ and $g_1 \sim g_2$ in $G$. We observe the { Cartesian} product is commutative and that $$d_{G \square H}\left((g_1, h_1), (g_2, h_2) \right)=d_G(g_1,g_2)+d_H(h_1,h_2).$$ 

The next result is a direct consequence of \cite[Theorem 4.1]{KM}. 

\begin{lemma}\label{cartGH}
	Let $G$ and $H$ denote connected NDB graphs with $|V(H)|\gamma_G=|V(G)|\gamma_H$. Then, the { Cartesian} product $G \square H$ is NDB with {$\gamma_{G\square H}=|V(H)|\cdot{\gamma_G}=|V(G)|\cdot{\gamma_H}$}. In particular, the { Cartesian} product of $n$-copies of $G$ is NDB with {$\gamma=|V(G)|^{n-1}\cdot{\gamma_G}$}. 
\end{lemma}

\begin{proof}
	Immediate from \cite[Theorem 4.1]{KM} and a straightforward induction argument. 
\end{proof}

\bigskip 

It was proved in \cite[Theorem 3.3]{KMMM05} that the { Cartesian} product of graphs is SDB if and only if both factors are SDB. Similarly, the { Cartesian} product of graphs is bipartite if and only if both factors are bipartite. Therefore the next results holds: 

\begin{lemma}\label{SDBGH}
	Let $G$ and $H$ denote connected graphs. Then, the { Cartesian} product $G \square H$ is SDB if and only if both $G$ and $H$ are SDB. In particular, the { Cartesian} product of $n$-copies of $G$ is SDB if and only if $G$ is SDB. 
\end{lemma}

\begin{lemma}\label{bipartiteGH}
	Let $G$ and $H$ denote connected graphs. Then, the { Cartesian} product $G \square H$ is bipartite if and only if both $G$ and $H$ are bipartite. In particular, the { Cartesian} product of $n$-copies of $G$ is bipartite if and only if $G$ is bipartite. 
\end{lemma}

We now show how the above results can be used to construct infinitely many examples of nonbipartite NDB graphs which are not SDB, provided that at least one such example exists.

\begin{proposition}\label{cart}
	Let $G$ denote a nonbipartite NDB graph which is not SDB. If $H$ is a NDB graph and {$|V(H)|\cdot{\gamma_G}=|V(G)|\cdot{\gamma_H}$} then the { Cartesian} product $G \square H$ is a nonbipartite NDB graph with {$\gamma_{G\square H}=|V(H)|\cdot{\gamma_G}=|V(G)|\cdot{\gamma_H}$} which is not SDB. In particular, the { Cartesian} product of $n$-copies of $G$ is a nonbipartite NDB graph with {$\gamma=|V(G)|^{n-1}\cdot{\gamma_G}$} that is not SDB. 
\end{proposition}

\begin{proof}
	Immediate from Lemmas \ref{cartGH}, \ref{SDBGH} and \ref{bipartiteGH}.
\end{proof}
\bigskip 

We will now construct an example of a nonbipartite NDB graph which is not SDB.

\begin{definition}\label{def:tricirc}
	Let $\G$ be the graph with vertex set $V=\{0,1,2\}\times \mathbb{Z}_{10}$ where the adjacencies are $(0,j) \sim (1,j+1)$, $(0,j) \sim (1,j+4)$, $(0,j) \sim (2,j+1)$, $(0,j) \sim (2,j+4)$, $(1,j) \sim (1,j+4)$ and $(2,j) \sim (2,j+4)$ for every $j\in \mathbb{Z}_{10}$  with all the computations in the second component performed modulo $10$. A graphical representation of $\G$ is shown in  Figure \ref{fig:tricirc}.
\end{definition}

{\small\begin{figure}[!ht]{\rm
			\begin{center}
				\begin{tikzpicture}[line cap=round,line join=round,>=triangle 45,x=1cm,y=1cm, scale=0.65]
					\fill[line width=1.5pt,color=ffffff,fill=ffffff,fill opacity=0.1] (4,8) -- (6.110298513851125,7.682709723376317) -- (8.215359754403861,8.033066842220464) -- (10.109125660608816,9.016775960658391) -- (11.606221230923719,10.537544776272947) -- (12.560100320047956,12.446509841572391) -- (12.87739059667164,14.556808355423515) -- (12.52703347782749,16.661869595976253) -- (11.543324359389565,18.555635502181207) -- (10.02255554377501,20.05273107249611) -- (8.113590478475565,21.006610161620348) -- (6.003291964624442,21.32390043824403) -- (3.8982307240717047,20.973543319399884) -- (2.00446481786675,19.98983420096196) -- (0.5073692475518472,18.469065385347402) -- (-0.44650984157239115,16.560100320047958) -- (-0.7638001181960741,14.449801806196833) -- (-0.4134429993519273,12.344740565644097) -- (0.570266119085999,10.45097465943914) -- (2.091034934700554,8.95387908912424) -- cycle;
					\draw [line width=1.5pt,color=ffffff] (4,8)-- (6.110298513851125,7.682709723376317);
					\draw [line width=1.5pt,color=ffffff] (6.110298513851125,7.682709723376317)-- (8.215359754403861,8.033066842220464);
					\draw [line width=1.5pt,color=ffffff] (8.215359754403861,8.033066842220464)-- (10.109125660608816,9.016775960658391);
					\draw [line width=1.5pt,color=ffffff] (10.109125660608816,9.016775960658391)-- (11.606221230923719,10.537544776272947);
					\draw [line width=1.5pt,color=ffffff] (11.606221230923719,10.537544776272947)-- (12.560100320047956,12.446509841572391);
					\draw [line width=1.5pt,color=ffffff] (12.560100320047956,12.446509841572391)-- (12.87739059667164,14.556808355423515);
					\draw [line width=1.5pt,color=ffffff] (12.87739059667164,14.556808355423515)-- (12.52703347782749,16.661869595976253);
					\draw [line width=1.5pt,color=ffffff] (12.52703347782749,16.661869595976253)-- (11.543324359389565,18.555635502181207);
					\draw [line width=1.5pt,color=ffffff] (11.543324359389565,18.555635502181207)-- (10.02255554377501,20.05273107249611);
					\draw [line width=1.5pt,color=ffffff] (10.02255554377501,20.05273107249611)-- (8.113590478475565,21.006610161620348);
					\draw [line width=1.5pt,color=ffffff] (8.113590478475565,21.006610161620348)-- (6.003291964624442,21.32390043824403);
					\draw [line width=1.5pt,color=ffffff] (6.003291964624442,21.32390043824403)-- (3.8982307240717047,20.973543319399884);
					\draw [line width=1.5pt,color=ffffff] (3.8982307240717047,20.973543319399884)-- (2.00446481786675,19.98983420096196);
					\draw [line width=1.5pt,color=ffffff] (2.00446481786675,19.98983420096196)-- (0.5073692475518472,18.469065385347402);
					\draw [line width=1.5pt,color=ffffff] (0.5073692475518472,18.469065385347402)-- (-0.44650984157239115,16.560100320047958);
					\draw [line width=1.5pt,color=ffffff] (-0.44650984157239115,16.560100320047958)-- (-0.7638001181960741,14.449801806196833);
					\draw [line width=1.5pt,color=ffffff] (-0.7638001181960741,14.449801806196833)-- (-0.4134429993519273,12.344740565644097);
					\draw [line width=1.5pt,color=ffffff] (-0.4134429993519273,12.344740565644097)-- (0.570266119085999,10.45097465943914);
					\draw [line width=1.5pt,color=ffffff] (0.570266119085999,10.45097465943914)-- (2.091034934700554,8.95387908912424);
					\draw [line width=1.5pt,color=ffffff] (2.091034934700554,8.95387908912424)-- (4,8);
					\draw [line width=1.5pt] (3.8982307240717047,20.973543319399884)-- (-0.7638001181960741,14.449801806196833);
					\draw [line width=1.5pt] (-0.7638001181960741,14.449801806196833)-- (4,8);
					\draw [line width=1.5pt] (4,8)-- (11.606221230923719,10.537544776272947);
					\draw [line width=1.5pt] (11.606221230923719,10.537544776272947)-- (11.543324359389565,18.555635502181207);
					\draw [line width=1.5pt] (11.543324359389565,18.555635502181207)-- (3.8982307240717047,20.973543319399884);
					\draw [line width=1.5pt] (8.113590478475565,21.006610161620348)-- (12.87739059667164,14.556808355423515);
					\draw [line width=1.5pt] (12.87739059667164,14.556808355423515)-- (8.215359754403861,8.033066842220464);
					\draw [line width=1.5pt] (8.215359754403861,8.033066842220464)-- (0.570266119085999,10.45097465943914);
					\draw [line width=1.5pt] (0.570266119085999,10.45097465943914)-- (0.5073692475518472,18.469065385347402);
					\draw [line width=1.5pt] (0.5073692475518472,18.469065385347402)-- (8.113590478475565,21.006610161620348);
					\draw [line width=1.5pt] (4.69184146819216,18.341401237899404)-- (9.334486065903286,16.913969253817122);
					\draw [line width=1.5pt] (9.334486065903286,16.913969253817122)-- (9.41157365540949,12.05745111492569);
					\draw [line width=1.5pt] (9.41157365540949,12.05745111492569)-- (4.816571808124001,10.483389822192684);
					\draw [line width=1.5pt] (4.816571808124001,10.483389822192684)-- (1.8996168986268076,14.36708458179952);
					\draw [line width=1.5pt] (1.8996168986268076,14.36708458179952)-- (4.69184146819216,18.341401237899404);
					\draw [line width=1.5pt] (7.2450641503782975,18.381928582061086)-- (10.16201905987549,14.49823382245425);
					\draw [line width=1.5pt] (10.16201905987549,14.49823382245425)-- (7.369794490310139,10.523917166354366);
					\draw [line width=1.5pt] (7.369794490310139,10.523917166354366)-- (2.7271498925990123,11.951349150436647);
					\draw [line width=1.5pt] (2.7271498925990123,11.951349150436647)-- (2.6500623030928065,16.80786728932808);
					\draw [line width=1.5pt] (2.6500623030928065,16.80786728932808)-- (7.2450641503782975,18.381928582061086);
					\draw [line width=1.5pt] (4.69184146819216,18.341401237899404)-- (6.003291964624442,21.32390043824403);
					\draw [line width=1.5pt] (6.003291964624442,21.32390043824403)-- (7.2450641503782975,18.381928582061086);
					\draw [line width=1.5pt] (7.2450641503782975,18.381928582061086)-- (10.02255554377501,20.05273107249611);
					\draw [line width=1.5pt] (10.02255554377501,20.05273107249611)-- (9.334486065903286,16.913969253817122);
					\draw [line width=1.5pt] (9.334486065903286,16.913969253817122)-- (12.52703347782749,16.661869595976253);
					\draw [line width=1.5pt] (12.52703347782749,16.661869595976253)-- (10.16201905987549,14.49823382245425);
					\draw [line width=1.5pt] (10.16201905987549,14.49823382245425)-- (12.560100320047956,12.446509841572391);
					\draw [line width=1.5pt] (12.560100320047956,12.446509841572391)-- (9.41157365540949,12.05745111492569);
					\draw [line width=1.5pt] (9.41157365540949,12.05745111492569)-- (10.109125660608816,9.016775960658391);
					\draw [line width=1.5pt] (10.109125660608816,9.016775960658391)-- (7.369794490310139,10.523917166354366);
					\draw [line width=1.5pt] (7.369794490310139,10.523917166354366)-- (6.110298513851125,7.682709723376317);
					\draw [line width=1.5pt] (6.110298513851125,7.682709723376317)-- (4.816571808124001,10.483389822192684);
					\draw [line width=1.5pt] (4.816571808124001,10.483389822192684)-- (2.091034934700554,8.95387908912424);
					\draw [line width=1.5pt] (2.091034934700554,8.95387908912424)-- (2.7271498925990123,11.951349150436647);
					\draw [line width=1.5pt] (2.7271498925990123,11.951349150436647)-- (-0.4134429993519273,12.344740565644097);
					\draw [line width=1.5pt] (-0.4134429993519273,12.344740565644097)-- (1.8996168986268076,14.36708458179952);
					\draw [line width=1.5pt] (1.8996168986268076,14.36708458179952)-- (-0.44650984157239115,16.560100320047958);
					\draw [line width=1.5pt] (-0.44650984157239115,16.560100320047958)-- (2.6500623030928065,16.80786728932808);
					\draw [line width=1.5pt] (2.6500623030928065,16.80786728932808)-- (2.00446481786675,19.98983420096196);
					\draw [line width=1.5pt] (2.00446481786675,19.98983420096196)-- (4.69184146819216,18.341401237899404);
					\draw [line width=1.5pt] (3.8982307240717047,20.973543319399884)-- (6.003291964624442,21.32390043824403);
					\draw [line width=1.5pt] (6.003291964624442,21.32390043824403)-- (8.113590478475565,21.006610161620348);
					\draw [line width=1.5pt] (8.113590478475565,21.006610161620348)-- (10.02255554377501,20.05273107249611);
					\draw [line width=1.5pt] (10.02255554377501,20.05273107249611)-- (11.543324359389565,18.555635502181207);
					\draw [line width=1.5pt] (11.543324359389565,18.555635502181207)-- (12.52703347782749,16.661869595976253);
					\draw [line width=1.5pt] (12.52703347782749,16.661869595976253)-- (12.87739059667164,14.556808355423515);
					\draw [line width=1.5pt] (12.87739059667164,14.556808355423515)-- (12.560100320047956,12.446509841572391);
					\draw [line width=1.5pt] (12.560100320047956,12.446509841572391)-- (11.606221230923719,10.537544776272947);
					\draw [line width=1.5pt] (11.606221230923719,10.537544776272947)-- (10.109125660608816,9.016775960658391);
					\draw [line width=1.5pt] (10.109125660608816,9.016775960658391)-- (8.215359754403861,8.033066842220464);
					\draw [line width=1.5pt] (8.215359754403861,8.033066842220464)-- (6.110298513851125,7.682709723376317);
					\draw [line width=1.5pt] (6.110298513851125,7.682709723376317)-- (4,8);
					\draw [line width=1.5pt] (4,8)-- (2.091034934700554,8.95387908912424);
					\draw [line width=1.5pt] (2.091034934700554,8.95387908912424)-- (0.570266119085999,10.45097465943914);
					\draw [line width=1.5pt] (0.570266119085999,10.45097465943914)-- (-0.4134429993519273,12.344740565644097);
					\draw [line width=1.5pt] (-0.4134429993519273,12.344740565644097)-- (-0.7638001181960741,14.449801806196833);
					\draw [line width=1.5pt] (-0.7638001181960741,14.449801806196833)-- (-0.44650984157239115,16.560100320047958);
					\draw [line width=1.5pt] (-0.44650984157239115,16.560100320047958)-- (0.5073692475518472,18.469065385347402);
					\draw [line width=1.5pt] (0.5073692475518472,18.469065385347402)-- (2.00446481786675,19.98983420096196);
					\draw [line width=1.5pt] (2.00446481786675,19.98983420096196)-- (3.8982307240717047,20.973543319399884);
					\begin{scriptsize}
						\draw [fill=uuuuuu] (4,8) circle (5.5pt);
						\draw[color=uuuuuu] (4,7.4) node { $(1,6)$};
						\draw [fill=uuuuuu] (6.110298513851125,7.682709723376317) circle (5.5pt);
						\draw[color=uuuuuu] (6.25,7) node { $(0,5)$};
						\draw [fill=uuuuuu] (8.215359754403861,8.033066842220464) circle (5.5pt);
						\draw[color=uuuuuu] (8.5,7.4) node { $(1,9)$};
						\draw [fill=uuuuuu] (10.109125660608816,9.016775960658391) circle (5.5pt);
						\draw[color=uuuuuu] (10.446724339151393,8.5) node { $(0,8)$};
						\draw [fill=uuuuuu] (11.606221230923719,10.537544776272947) circle (5.5pt);
						\draw[color=uuuuuu] (12,10) node { $(1,2)$};
						\draw [fill=uuuuuu] (12.560100320047956,12.446509841572391) circle (5.5pt);
						\draw[color=uuuuuu] (13.2,12) node { $(0,1)$};
						\draw [fill=uuuuuu] (12.87739059667164,14.556808355423515) circle (5.5pt);
						\draw[color=uuuuuu] (13.5,15) node {$(1,5)$};
						\draw [fill=uuuuuu] (12.52703347782749,16.661869595976253) circle (5.5pt);
						\draw[color=uuuuuu] (13.2,17) node { $(0,4)$};
						\draw [fill=uuuuuu] (11.543324359389565,18.555635502181207) circle (5.5pt);
						\draw[color=uuuuuu] (12,19) node { $(1,8)$};
						\draw [fill=uuuuuu] (10.02255554377501,20.05273107249611) circle (5.5pt);
						\draw[color=uuuuuu] (10.5,20.5) node { $(0,7)$};
						\draw [fill=uuuuuu] (8.113590478475565,21.006610161620348) circle (5.5pt);
						\draw[color=uuuuuu] (8.5,21.5) node { $(1,1)$};
						\draw [fill=uuuuuu] (6.003291964624442,21.32390043824403) circle (5.5pt);
						\draw[color=uuuuuu] (6.25,22) node { $(0,0)$};
						\draw [fill=uuuuuu] (3.8982307240717047,20.973543319399884) circle (5.5pt);
						\draw[color=uuuuuu] (4,21.5) node { $(1,4)$};
						\draw [fill=uuuuuu] (2.00446481786675,19.98983420096196) circle (5.5pt);
						\draw[color=uuuuuu] (1.5,20.5) node { $(0,3)$};
						\draw [fill=uuuuuu] (0.5073692475518472,18.469065385347402) circle (5.5pt);
						\draw[color=uuuuuu] (0,19) node { $(1,7)$};
						\draw [fill=uuuuuu] (-0.44650984157239115,16.560100320047958) circle (5.5pt);
						\draw[color=uuuuuu] (-1,17) node { $(0,6)$};
						\draw [fill=uuuuuu] (-0.7638001181960741,14.449801806196833) circle (5.5pt);
						\draw[color=uuuuuu] (-1.5,15) node { $(1,0)$};
						\draw [fill=uuuuuu] (-0.4134429993519273,12.344740565644097) circle (5.5pt);
						\draw[color=uuuuuu] (-1,12) node { $(0,9)$};
						\draw [fill=uuuuuu] (0.570266119085999,10.45097465943914) circle (5.5pt);
						\draw[color=uuuuuu] (0,10) node { $(1,3)$};
						\draw [fill=uuuuuu] (2.091034934700554,8.95387908912424) circle (5.5pt);
						\draw[color=uuuuuu] (1.5,8.5) node { $(0,2)$};
						\draw [fill=uuuuuu] (4.69184146819216,18.341401237899404) circle (5.5pt);
						\draw[color=uuuuuu] (4.2,19.212739137375625) node { $(2,4)$};
						\draw [fill=uuuuuu] (7.2450641503782975,18.381928582061086) circle (5.5pt);
						\draw[color=uuuuuu] (7.569282903671776,19.25326648153731) node { $(2,1)$};
						\draw [fill=uuuuuu] (2.6500623030928065,16.80786728932808) circle (5.5pt);
						\draw[color=uuuuuu] (1.8,17.2) node { $(2,7)$};
						\draw [fill=uuuuuu] (1.8996168986268076,14.36708458179952) circle (5.5pt);
						\draw[color=uuuuuu] (1.1,14.5) node { $(2,0)$};
						\draw [fill=uuuuuu] (2.7271498925990123,11.951349150436647) circle (5.5pt);
						\draw[color=uuuuuu] (2,11.5) node { $(2,3)$};
						\draw [fill=uuuuuu] (4.816571808124001,10.483389822192684) circle (5.5pt);
						\draw[color=uuuuuu] (4.5,9.8) node { $(2,6)$};
						\draw [fill=uuuuuu] (7.369794490310139,10.523917166354366) circle (5.5pt);
						\draw[color=uuuuuu] (7.7,9.8) node { $(2,9)$};
						\draw [fill=uuuuuu] (9.41157365540949,12.05745111492569) circle (5.5pt);
						\draw[color=uuuuuu] (10.1,11.5) node { $(2,2)$};
						\draw [fill=uuuuuu] (10.16201905987549,14.49823382245425) circle (5.5pt);
						\draw[color=uuuuuu] (11,14.5) node { $(2,5)$};
						\draw [fill=uuuuuu] (9.334486065903286,16.913969253817122) circle (5.5pt);
						\draw[color=uuuuuu] (10,17.3) node { $(2,8)$};
					\end{scriptsize}
				\end{tikzpicture}
				\caption{\rm 
				A regular nonbipartite NDB graph $\G$ that is not SDB.
				}
				\label{fig:tricirc}
			\end{center}
}\end{figure}
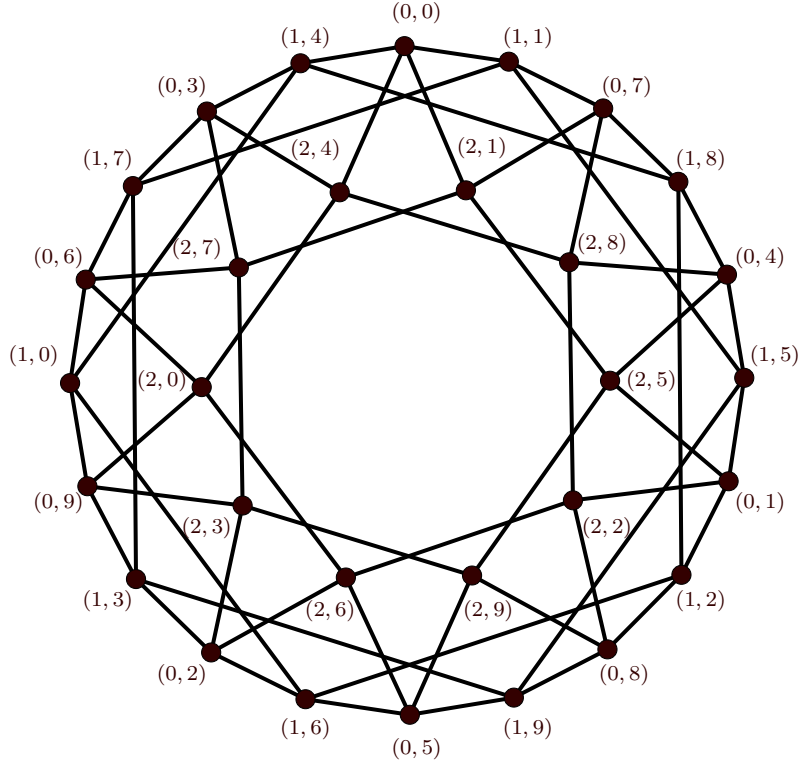}


Keeping in mind the graph $\G$ defined in Definition \ref{def:tricirc}, we now consider certain maps on $V(\G)$. Let $\rho, \ \tau$ and $\varphi$ the functions such that for every $j \in \mathbb{Z}_{10}$, 
\begin{eqnarray}
	\rho(0,j)&=&(0,j+7), \hspace{0.5cm} \rho(1,j)=(1,j+7), \hspace{0.5cm} \rho(2,j)=(2,j+7), \nonumber \\ 
	\tau(0,j)&=&(0,7-j), \hspace{0.5cm} \tau(1,j)=(1,2-j), \hspace{0.5cm} \tau(2,j)=(2,2-j), \nonumber \\ 
	\varphi(0,j)&=&(0,j), \hspace{1.18cm} \varphi(1,j)=(2,j), \hspace{1.18cm} \varphi(2,j)=(1,j), \nonumber 
\end{eqnarray}
with all the computations in the second component performed modulo $10$. 
It is easy to see that these maps are automorphisms of $\G$. Moreover, we observe $\rho$ is a rotation, $\tau$ is a reflection and $\varphi$ swaps vertices with $1$ and $2$ as first coordinate and fixes all the others. 

\begin{proposition}\label{ex:reg}
	Let the graph $\G$ be as defined in Definition \ref{def:tricirc}.
	Then,  $\Gamma$ is a regular nonbipartite NDB graph that is not SDB. 
\end{proposition}

\begin{proof}
	Let the graph $\G$ be as defined in Definition \ref{def:tricirc}. See also Figure \ref{fig:tricirc}. Notice that $\G$ has diameter $4$.  By construction we observe every vertex in $\G$ has valency $4$ and that $\G$ has odd cycles. Therefore, $\G$ is a regular nonbipartite graph. Let Aut$(\G)$ denote the automorphism group of $\G$. For $\alpha \in$ Aut$(\G)$ and every pair of adjacent  vertices $u,v \in V(\G)$ we have $\alpha\left( W_{u,v}\right)=W_{\alpha(u), \alpha(v)} $ and since $\alpha$ is a bijection, $|W_{\alpha(u), \alpha(v)}|=|W_{u,v}|$. Pick now the edge $(0,0)(1,1)$ and note the following hold: 
	\begin{eqnarray}
		W_{(0,0),(1,1)}&=&\{(0,0), (2,1), (2,4), (1,4), (1,0), (2,0), (2,5), (2,7), (1,6), (2,3), (2,9), (2,6)\} , \nonumber \\
		W_{(1,1),(0,0)}&=&\{(1,1), (1,7), (1,5), (0,7), (1,3), (0,4), (1,9), (0,6), (0,1), (0,2),  (0,5), (0,8) \}. \nonumber 
	\end{eqnarray}
	Then, the edge 	$(0,0)(1,1)$  is balanced and $|W_{(0,0),(1,1)}|=|W_{(1,1),(0,0)}|=12$. Furthermore, the automorphism $\tau$ maps the edge $(0,0)(1,1)$ to the edge $(0,7)(1,1)$ and so, $(0,7)(1,1)$ is balanced and $|W_{(0,7),(1,1)}|=12$. Considering $\varphi \in $ Aut$(\G)$ we also observe the edges $(0,0)(1,1)$ and $(0,7)(1,1)$ are respectively mapped to the edges $(0,0)(2,1)$ and $(0,7)(2,1)$ which shows the edges $(0,0)(2,1)$ and $(0,7)(2,1)$ are balanced and $|W_{(0,0),(2,1)}|=|W_{(0,7),(2,1)}|=12$. Therefore, since $\rho$ is an automorphism of $\G$, it follows from the above comments that all the edges $(0,j)(1,j+1)$, $(0,j)(1,j+4)$, $(0,j)(2,j+1)$, $(0,j)(2,j+4)$ are all balanced and 
	$$|W_{(0,j),(1,j+1)}|=|W_{(0,j),(1,j+4)}|=|W_{(0,j),(2,j+1)}|=|W_{(0,j),(2,j+4)}|=12$$ 
	for every $j \in \mathbb{Z}_{10}$. 	Pick now the edge $(1,1)(1,5)$ and note 
	\begin{eqnarray}
		W_{(1,1),(1,5)}&=&\{(1,1), (1,7), (0,0), (0,7), (0,3), (2,4), (2,1), (1,4), (0,6), (1,0),  (2,0), (2,7) \} ,\nonumber \\
		W_{(1,5),(1,1)}&=&\{(1,5), (0,4), (1,9), (0,1), (2,2), (1,2), (0,5), (2,5), (0,8), (1,6), (2,9), (2,6)\}, \nonumber
	\end{eqnarray}
	which shows  this edge is balanced and $|W_{(1,1),(1,5)}|=|W_{(1,5),(1,1)}|=12$. Since $\rho \in $ Aut$(\G)$, it is easy to see there exists an automorphism of $\G$ that maps the edge $(1,1)(1,5)$ to the edge $(1,j)(1,j+4)$ and as  $\varphi \in$ Aut$(\G)$ swaps vertices with $1$ and $2$ as first coordinate and fixes all the others, that there exists an automorphism of $\G$ that maps the edge $(1,1)(1,5)$ to the edge $(2,j)(2,j+4)$. We thus have the edges $(1,j)(1,j+4)$ and $(2,j)(2,j+4)$ are all balanced and $|W_{(1,j),(1,j+4)}|=|W_{(2,j),(2,j+4)}|=12$. 
	Hence, $\G$ is NDB with $\gamma=12$. We also notice 
	\begin{eqnarray}
		D^2_3((1,1),(0,0))&=&\{(1,3), (0,4), (1,9), (0,6), (0,1)\}, \nonumber \\
		D^3_2((1,1),(0,0))&=&\{(1,0), (2,0), (2,5), (2,7)\}. \nonumber	
	\end{eqnarray}
	This yields $\G$ is not SDB. The result follows. 
\end{proof}

\bigskip 

The graph given in Definition~\ref{def:tricirc} can be used to construct an infinite family of regular nonbipartite NDB graphs which are not SDB.

\begin{corollary}\label{regular}
	There exists infinitely many regular nonbipartite NDB graphs which are not SDB.
\end{corollary} 

\begin{proof}
	Let the graph $\G$ be as defined in Definition \ref{def:tricirc} and consider the { Cartesian} product of $n$ copies of $\G$. The result now is a straightforward consequence of Propositions \ref{cart} and \ref{ex:reg}.  
\end{proof}

\bigskip

Corollary \ref{regular} provides an infinite family of nonbipartite regular NDB graphs which are not SDB. We next give a construction of a nonregular infinite family. 

\begin{definition}\label{main:defk}
	Let $k\geq 3$ be an integer. Let $\G^{(k)}$ denote the graph of order $12k+6$ with vertex set $V_k=\left\lbrace x_i \ | \ i \in \mathbb{Z}_{8k+4}\right\rbrace \cup  \left\lbrace y_i \ | \ i \in \mathbb{Z}_{4k+2}\right\rbrace$ where $x_i \sim x_{i+1}$ and $y_i \sim x_{i+m}$ with $m\in \{0,2k-1,2k+1,4k+2,6k+1,6k+3\}$. All the computations in the index of $x_j$ are performed modulo $8k+4$ while all the computations in the index of $y_j$ are performed modulo $4k+2$. 
\end{definition}

Throughout this section we will need the following notation. 

\begin{notation}
	With reference to Definition \ref{main:defk}, {for an integer $k\geq 3$}, any subset $X \subseteq V_k$ will be { identified with a pair of sets $(A, B)$ where $A$ is the set of indexes of $x_i$ vertices that belong to $X$, while $B$ is the set of indexes of $y_i$ vertices that belong to $X$, that is $A=\{i \in \mathbb{Z}_{8k+4} \ | \  x_i \in X\}$ and $B=\{i \in \mathbb{Z}_{4k+2} \ | \  y_i \in X\}$}. Let $\ell \in \{4k+2, 8k+4\}$ and let $H\subseteq \mathbb{Z}_{\ell}$. For any integer $j$, we denote $j+H=\{j+h \ | \ h\in H\}$ where the computations are performed modulo $\ell$. Moreover, for $h\in H$ we denote $\left\langle h\right\rangle =\{nh: n\in \mathbb{Z}\}$ and  $\left\langle h\right\rangle^* =\left\langle h\right\rangle\setminus \{0\}$.
\end{notation}

The following results will be very useful in the rest of the paper. 

\begin{lemma}\label{lemma:neighbours}
	For an integer $k\geq 3$, let the graph $\G^{(k)}$ be as defined in Definition \ref{main:defk}. Let $K=\{0, 2k+1, 2k+3\}$ and $M=\{0,2k-1,2k+1,4k+2,6k+1,6k+3\}$. The following holds:  
	\begin{enumerate}[label=(\roman*)]
		\item $S_0(x_j)=\left( \left\lbrace j \right\rbrace, \emptyset \right)  $ and $S_1(x_j)=\left( \left\lbrace {j\pm1} \right\rbrace, j+K\right)  $ for $x_j \in V_k$. {In particular}, $|S_0(x_j)|=1$ and $|S_1(x_j)|=5$.
		\item $S_0(y_j)=\left( \emptyset, \left\lbrace j \right\rbrace \right)  $ and $S_1(y_j)=\left( j+M, \emptyset\right)  $ for $y_j \in V_k$. {In particular}, $|S_0(y_j)|=1$ and $|S_1(y_j)|=6$. 
		\end{enumerate}
	\end{lemma}

\begin{proof}
	Pick $x_j, y_j \in V_k$. It is clear that $S_0(x_j)=\left\lbrace x_j \right\rbrace $ and  $S_0(y_j)=\left\lbrace y_j \right\rbrace $. By Definition \ref{main:defk} we observe $\{x_{j-1}, x_{j+1}\} \subseteq S_1(x_j)$ and $x_j \sim y_{j+m}$ with $m\in \{0,2k+1,2k+3\}$. Similarly, vertex  $y_j \sim x_{j+m}$ with $m\in \{0,2k-1,2k+1,4k+2,6k+1,6k+3\}$. The result follows.
\end{proof}

\begin{lemma}\label{lemma1k}
	For an integer $k\geq 5$, let the graph $\G^{(k)}$ be as defined in Definition \ref{main:defk}.  For $x_j \in V_k$ the following hold:  
	\begin{enumerate}[label=(\roman*)]
		\item $	S_2(x_j) =\left( \pm 2+j+\left\langle 2k+1 \right\rangle \cup j+\left\langle 2k+1 \right\rangle^*, \pm 1+j+K\right) $,
		\item $	S_3(x_j) =\left( \pm 3+j+\left\langle 2k+1 \right\rangle \cup \pm 1+j+\left\langle 2k+1 \right\rangle^*, \pm 3+j+1+\left\langle 2k+1 \right\rangle \cup\{j+2\}\right) $,
		\item $	S_4(x_j) =\left( \pm 4+j+\left\langle 2k+1 \right\rangle, \pm 4+j+1+\left\langle 2k+1 \right\rangle \cup\{j+3\}\right) $ ,
		\item $S_i(x_j) =\left(\pm i+j+\left\langle 2k+1 \right\rangle , \pm i+j+1+\left\langle 2k+1 \right\rangle  \right) $, for every $i \in \{5,\ldots, k\}$, 	
		\item $|S_2(x_j)|=16$, $|S_3(x_j)|=19$, $|S_4(y_j)|=13$ and $|S_i(x_j)|=12$ for every $5\leq i \leq k$. Moreover, the eccentricity of $x_j$ equals $k$.
	\end{enumerate}
\end{lemma}

\begin{proof}
	Pick a vertex $x_j \in V_k$. Assume for a moment
	that $z\in S_i(x_j)$ for some $0 \leq i \leq \epsilon(x_j)$ and let $w$ be a neighbour of $z$. Then, by the triangle inequality, $d(x_j,w) \in \{i-1,i,i+1\}$ and so $w\in S_{i-1}(x_j)\cup S_{i}(x_j) \cup S_{i+1}(x_j)$. {Therefore,} $S_{i+1}(x_j)$ consists of all the neighbours of vertices in $S_{i}(x_j)$ which are not in $S_{i-1}(x_j)$ nor $S_{i}(x_j)$. Now, $(i)$--$(iii)$ immediately {follow} from Lemma \ref{lemma:neighbours} and the above comments after a careful inspection of the neighbours’ sets of vertices in $S_i(x_j)$. We now prove part $(iv)$ by induction. Similarly as above we see that $(iv)$ holds for $i \in \{5,6\}$. Let us now assume that $(iv)$ holds for $i-1$ and $i$, where $i\geq6$. Hence, we have
	\begin{eqnarray}
		S_{i-1}(x_j) &=&\left(\pm (i-1)+j+\left\langle 2k+1 \right\rangle , \pm (i-1)+j+1+\left\langle 2k+1 \right\rangle  \right), \nonumber \\ 
		S_i(x_j) &=&\left(\pm i+j+\left\langle 2k+1 \right\rangle , \pm i+j+1+\left\langle 2k+1 \right\rangle  \right) \label{si}.
	\end{eqnarray}
	Next, we compute the neighbours of the vertices belonging to the set $S_i(x_j)$. By Lemma \ref{lemma:neighbours} and equation \eqref{si}, we get
	\begin{eqnarray}
		S\left( \left( \pm i+j+\left\langle 2k+1 \right\rangle, \emptyset \right) \right)  &=&\left(\pm i \pm 1+j+\left\langle 2k+1 \right\rangle, \pm i+j+\left\langle 2k+1 \right\rangle +K \right),\label{bn4} \\ 
		S\left( \left(  \emptyset, \pm i+j+1+\left\langle 2k+1 \right\rangle \right) \right)  &=&\left(\pm i+j+1+\{0, 2k+1\}+M, \emptyset\right) \label{bn5}, 
	\end{eqnarray}
	where $K$ and $M$ are the sets as defined in Lemma \ref{lemma:neighbours}.  Observe  that
	\begin{equation}
		\left\langle 2k+1 \right\rangle +K =\left\langle 2k+1 \right\rangle \cup (2+\left\langle 2k+1 \right\rangle), \label{bn6}
	\end{equation}
	where the operations are performed modulo $4k+2$. Similarly, we have 
	\begin{equation}
		\{0, 2k+1\}+M =(-2+\left\langle 2k+1 \right\rangle) \cup \left\langle 2k+1 \right\rangle, \label{bn7}
	\end{equation}
	where the operations are performed modulo $8k+4$.  Therefore, from \eqref{bn4}--\eqref{bn7} it turns out the set of all neighbours of the vertices which are in $S_i(x_j)$ is given as follows: 
	\begin{eqnarray}
		S\left( S_i(x_j)\right) 
		&=&	 \left(\pm i \pm 1+j+\left\langle 2k+1 \right\rangle , \pm i\pm 1+j+1+\left\langle 2k+1 \right\rangle  \right). \label{bn8} \nonumber
	\end{eqnarray}
	We thus have 
	\begin{eqnarray*}
	S_{i+1}(x_j)&=& S\left( S_i(x_j)\right)  \setminus \left( S_{i-1}(x_j)\cup S_i(x_j)\right)\\
	&=&  \left(\pm (i + 1)+j+\left\langle 2k+1 \right\rangle , \pm (i+ 1)+j+1+\left\langle 2k+1 \right\rangle  \right) 
	\end{eqnarray*}	
	proving the claim $(iv)$. 
	
	Let us now prove $(v)$. The first part of the statement immediately holds from  $(i)$--$(iv)$ above. To prove the second part, let $\ell$ denote the eccentricity of $x_j$. From Lemma  \ref{lemma:neighbours} and $(i)$--$(iv)$ above, the sets $S_i(x_j) \ (0 \leq i \leq k)$ are nonempty and so, $\ell \geq k$. 
	Observe that $\sum_{i=0}^{k} |S_i(x_j)|=12k+6=|V_k|$.
	Since the collection of all the sets  $S_i(x_j) \ (0 \leq i \leq \ell)$ is a partition of the vertex set it follows that the sets  $S_i(x_j)$ are empty for $i>k$. Then, $\ell\leq k$ and the result follows. 
\end{proof}

\bigskip 
The proof of the next result can be done in a similar way to that of Lemma \ref{lemma1k} above and is therefore omitted and left to the reader.
\bigskip 

\begin{lemma}\label{lemma2k}
	For an integer $k\geq 5$, let the graph $\G^{(k)}$ be as defined in Definition \ref{main:defk}. For $y_j \in V_k$ the following hold:  
	\begin{enumerate}[label=(\roman*)]
		
		\item $	S_2(y_j) =\left( \pm 1+j+\left\langle 2k+1 \right\rangle \cup j+\{2k-2, 6k\}, \pm 2+j+\left\langle 2k+1 \right\rangle \cup j+\{2k+1\}\right) $.
		\item $	S_3(y_j) =\left( \pm 3+j-1+\left\langle 2k+1 \right\rangle \cup -2+j+\left\langle 4k+2 \right\rangle , \pm 3+j+\left\langle 2k+1 \right\rangle \cup \pm 1+j+\left\langle 2k+1\right\rangle \right) $.
		\item $	S_4(y_j) =\left( \pm 4+j-1+\left\langle 2k+1 \right\rangle \cup -3+j+\left\langle 4k+2 \right\rangle , \pm 4+j+\left\langle 2k+1 \right\rangle \right) $.
		\item For every $5\leq i \leq k$, the set $	S_i(y_j) =\left( \pm i+j-1+\left\langle 2k+1 \right\rangle , \pm i+j+\left\langle 2k+1 \right\rangle \right) $.
		\item $|S_2(y_j)|=15$, $|S_3(y_j)|=18$, $|S_4(y_j)|=14$ and $|S_i(y_j)|=12$ for every $5\leq i \leq k$. Moreover, the eccentricity of $y_j$ equals $k$.
	\end{enumerate}
\end{lemma}

For an integer $k\geq 5$, let the graph $\G^{(k)}$ be as defined in Definition \ref{main:defk}. We next show that some edges of  $\G^{(k)}$ are balanced.

\begin{lemma}\label{propxj}
	For an integer $k\geq 5$, let the graph $\G^{(k)}$ be as defined in Definition \ref{main:defk}.  For the edge $x_jx_{j+1}$ the following hold:  
	\begin{enumerate}[label=(\roman*)]
		\item $|D^1_0(x_j, x_{j+1})|=|D^0_1(x_j, x_{j+1})|=1$.
		\item $|D^2_1(x_j, x_{j+1})|=|D^1_2(x_j, x_{j+1})|=4$.
		\item $|D^3_2(x_j, x_{j+1})|=|D^2_3(x_j, x_{j+1})|=12$.
		\item $|D^4_3(x_j, x_{j+1})|=|D^3_4(x_j, x_{j+1})|=7$.	
		\item $|D^{\ell+1}_{\ell}(x_j, x_{j+1})|=|D^{\ell}_{\ell+1}(x_j, x_{j+1})|=6$ for all $4\leq \ell \leq k-1$.
		\item $|D^k_k(x_j, x_{j+1})|=6$. 
		\item The edge $x_jx_{j+1}$ is balanced and the sets $D^i_i(x_j, x_{j+1}) \ (1\leq i \leq k-1)$ are all empty.
	\end{enumerate}
\end{lemma}

\begin{proof}
	Pick $j \in \mathbb{Z}_{8k+4} $ and consider the edge $x_jx_{j+1}$.  By Lemma \ref{lemma1k} and Lemma \ref{lemma2k} we first observe that $\G^{(k)}$ has diameter $k$. Now, $(i)$--$(vi)$ immediately follows from Lemma \ref{lemma1k}. Let us now prove $(vii)$. From $(i)$--$(v)$ above, we notice 
	\begin{equation}
		|W_{x_j, x_{j+1}}|=	\sum_{i=0}^{k-1}|D^{i+1}_{i}(x_j, x_{j+1})|=6k=\sum_{i=0}^{k-1}|D^{i}_{i+1}(x_j, x_{j+1})|=|W_{ x_{j+1}, x_j}|. \nonumber 
	\end{equation}
	Hence, the edge $x_jx_{j+1}$ is balanced. Moreover, by $(vi)$ above we also notice $$\sum_{i=1}^{k-1}|D^{i}_{i}(x_j, x_{j+1})|=|V_k|-2|W_{x_j, x_{j+1}}|-|D^{k}_{k}(x_j, x_{j+1})|=0.$$ The result follows. 
\end{proof}

\bigskip 

The proof of the next result {is} omitted as { it} can be carried out using the same arguments as the proof of Lemma \ref{propxj}.

\begin{lemma}\label{propxjyell}
For an integer $k\geq 5$, let the graph $\G^{(k)}$ be as defined in Definition \ref{main:defk} and let $K=\{0,2k+1,2k+3\}$. For every $\ell\in K$ and for every edge $x_jy_{\ell}$ the following hold:  
\begin{enumerate}[label={\it (\roman*)}]
\item $|D^1_0(x_j, y_{\ell})|=|D^0_1(x_j, y_{\ell})|=1.$
\item $|D^2_1(x_j, y_{\ell})|=5$ and $|D^1_2(x_j, y_{\ell})|=4$.
\item $|D^3_2(x_j, y_{\ell})|=|D^2_3(x_j, y_{\ell})|=11$.
\item $|D^4_3(x_j, y_{\ell})|=7$ and $|D^3_4(x_j, y_{\ell})|=8$.
\item $|D^{i+1}_{i}(x_j, y_{\ell})|=|D^{i}_{i+1}(x_j, y_{\ell})|=6$ for all $4\leq i \leq k-1$.
\item $|D^k_k(x_j, y_{\ell})|=6$.
\item The edge $x_j y_{\ell}$ is balanced and the sets $D^i_i(x_j,  y_{\ell}) \ (1\leq i \leq k-1)$ are all empty.
\end{enumerate}
\end{lemma}

We are now ready to provide  an infinite family of nonbipartite and nonregular NDB graphs which are not SDB. 

\begin{proposition}\label{mainprop}
	For an integer $k\geq 5$, let the graph $\G^{(k)}$ be as defined in Definition \ref{main:defk}.  Then, $\G^{(k)}$ is a nonbipartite NDB graph which is not SDB nor regular. 
\end{proposition}

\begin{proof}
	 By Definition \ref{main:defk} and Lemma \ref{lemma:neighbours}, it is clear that   $\G^{(k)}$ is not regular. { This implies that  $\G^{(k)}$ is not SDB since for at least one edge $uv$ the corresponding sets $D^1_2(u,v)$ and $D^2_1(u,v)$ will not be of the same cardinality.}  Pick $j\in \mathbb{Z}_{8k+4}$. Recall that $\{x_{j-1}, x_{j+1}\} \subseteq S_1(x_j)$ and $x_j \sim y_{j+m}$ with $m\in \{0,2k+1,2k+3\}$. It now follows from Lemma \ref{propxjyell}  that the edges $x_jx_{j+1}$, $x_jy_{j}$, $x_j  y_{2k+1+j}$ and $x_jy_{2k+3+j}$ are all balanced. Moreover, it turns out that $$|W_{x_j, x_{j+1}}|=|W_{x_j, y_{j}}|=|W_{x_j, y_{2k+j+1}}|=|W_{x_j, x_{2k+3+j}}|=6k.$$
	 In addition, for $i, i' \in \mathbb{Z}_{4k+2}$ we observe vertices $y_i$ and $y_{i'}$ are not adjacent. Since $j$ is arbitrary, we thus have all the edges of  $\G^{(k)}$ are balanced. Consequently, it follows from the above comments that  $\G^{(k)}$  is NDB with $\gamma=6k$. We also notice $\G^{(k)}$ is nonbipartite as the set $D^k_k(x_j,  y_{j})$ is nonempty by Lemma \ref{propxjyell}. This concludes the proof. 
\end{proof}

We end this section with following two remarks. 

\begin{remark}
Graphs $\G^{(3)}$ and $\G^{(4)}$ are also nonbipartite NDB graphs which are not SDB, with $\gamma=18$ and $\gamma=24$ respectively, but we considered only the case when $k\geq 5$ for the simplicity of proofs.   
\end{remark}

\begin{remark}
Graphs $\G^{(k)}$ defined in Definition \ref{main:defk} are prime with respect to the { Cartesian} product of graphs (cannot be obtained as a { Cartesian} product of two non-trivial graphs). 
{ 
Suppose $\G^{(k)}\cong G \square H$ for some graphs $G$ and $H$.
Observe that the edge $x_ix_{i+1}$ lies on exactly 2 cycles of length 4 in $\G^{(k)}$ for every $i\in \mathbb{Z}_{8k+4}$.  Since the vertices of $\Gamma^{(k)}$ have degree 5 or 6, without loss of generality we may assume that the minimum degree in $G$ is at least 3. It follows that the edge $x_{i}x_{i+1}$ must belong to the $H$-layers in the Cartesian product $G\square H$, since it lies only on 2 cycles of length 4. Then, it holds that all of the $x_i$ vertices belong to the same $H$-layer, implying that $H$ has at least $8k+4$ vertices. Since $|V(\Gamma^{(k)})|=12k+6=|V(G)|\cdot |V(H)|$, it follows that $G$ is the graph with one vertex.
}
\end{remark}

\section{Counterexamples to a conjecture regarding SDB graphs}\label{sec:Counterexample to conjecture Sparl}

Let $\G$ be a graph, and let $S$ be a subset of its vertex set. For a vertex $v$ of $\G$ we define 
$$d(v,S)=\sum_{x\in S}d(v,x).$$
Balakrishnan  {et al.}\  \cite{BCPSSS} proved that a connected graph $\G$ is distance-balanced if and only if $d(v,V(\G))=d(u,V(\G)))$ for all $u,v\in V(\G)$.
They posed the following conjecture regarding a similar characterization of strongly distance-balanced graphs.

\begin{conjecture}\label{conjecture Sparl}\cite[Conjecture 3.2]{BCPSSS} A graph $\G$ is strongly distance-balanced if and only if $d(u,W_{u,v})= d(v,W_{v,u})$ holds for
	every pair of adjacent vertices $u, v$ of $\G$.
\end{conjecture}

It is clear that strongly distance-balanced graphs satisfy the above condition, but the question was if the converse also holds. We will now provide an {infinite} family of counterexamples to Conjecture~\ref{conjecture Sparl}.

Let $k$ and $l$ be positive integers. Let $C_6(k,l)$ denote the graph obtained from the 6-cycle by replacing every vertex in one bipartition set of $C_6$ with $k$ pairwise non-adjacent vertices, and replacing every vertex in the other bipartition set of $C_6$ with $l$ pairwise non-adjacent vertices, see Figure~\ref{001} for an example.
To be more precise, let $\{x_0,x_1,x_2,x_3,x_4,x_5\}$ be the vertex set of the $6$-cycle, and let the vertex-set of $C_6(k,l)$ be $(\{x_0,x_2,x_4\}\times \mathbb{Z}_k) \cup (\{x_1,x_3,x_5\}\times \mathbb{Z}_l)$, and adjacencies given by $(x_{2i},r)\sim (x_{2i\pm 1},s)$ for every $i\in \{0,1,2\}$ and every $r\in \mathbb{Z}_k$, $s\in \mathbb{Z}_l$.   Observe that any permutation of vertices inside sets $\{x_{2i}\}\times \mathbb{Z}_{k}$ and $\{x_{2i+1}\} \times \mathbb{Z}_{l}$, preserves all the edges. Hence, {it} is an automorphism. Observe also that the 2-step rotation, function mapping $(x_i,j)$ into $(x_{i+2},j)$ {it} is also an automorphism of $C_6(k,l)$. 
It follows that the graph $C_6(k,l)$ is edge-transitive. Observe that $C_6(k,l)$ is vertex-transitive if and only if $k=l$. 

{\small\begin{figure}[!ht]{\rm
			\begin{center}
				\begin{tikzpicture}[line cap=round,line join=round,>=triangle 45,x=0.8cm,y=0.8cm, scale=0.55]
					\draw [line width=2pt] (-3.2963820752829553,-6.937153039137263)-- (-5.501742400013352,-10.756949170566905);
					\draw [line width=2pt] (-3.2963820752829553,-6.937153039137263)-- (-7,-10.803847577293364);
					\draw [line width=1.5pt] (-7,-10.803847577293364)-- (-4,-5.607695154586734);
					\draw [line width=1.5pt] (-4,-5.607695154586734)-- (2,-5.607695154586736);
					\draw [line width=1.5pt] (2,-5.607695154586736)-- (5,-10.80384757729337);
					\draw [line width=1.5pt] (2,-5.607695154586736)-- (3.3196988989082326,-10.75694917056691);
					\draw [line width=1.5pt] (1.1143385741778362,-6.937153039137263)-- (5,-10.80384757729337);
					\draw [line width=1.5pt] (1.1143385741778362,-6.937153039137263)-- (3.3196988989082326,-10.75694917056691);
					\draw [line width=1.5pt] (3.3196988989082326,-10.75694917056691)-- (1.114338574177835,-14.576745301996553);
					\draw [line width=1.5pt] (3.3196988989082326,-10.75694917056691)-- (2,-16);
					\draw [line width=1.5pt] (2,-5.607695154586736)-- (-3.2963820752829553,-6.937153039137263);
					\draw [line width=1.5pt] (1.1143385741778362,-6.937153039137263)-- (-4,-5.607695154586734);
					\draw [line width=1.5pt] (-3.2963820752829553,-6.937153039137263)-- (1.1143385741778362,-6.937153039137263);
					\draw [line width=1.5pt] (-4,-5.607695154586734)-- (-5.501742400013352,-10.756949170566905);
					\draw [line width=1.5pt] (-5.501742400013352,-10.756949170566905)-- (-3.2963820752829562,-14.576745301996551);
					\draw [line width=2pt] (-5.501742400013352,-10.756949170566905)-- (-4,-16);
					\draw [line width=1.5pt] (-7,-10.803847577293364)-- (-3.2963820752829562,-14.576745301996551);
					\draw [line width=1.5pt] (-7,-10.803847577293364)-- (-4,-16);
					\draw [line width=1.5pt] (-4,-16)-- (1.114338574177835,-14.576745301996553);
					\draw [line width=1.5pt] (-4,-16)-- (2,-16);
					\draw [line width=1.5pt] (-3.2963820752829562,-14.576745301996551)-- (1.114338574177835,-14.576745301996553);
					\draw [line width=1.5pt] (-3.2963820752829562,-14.576745301996551)-- (2,-16);
					\draw [line width=1.5pt] (1.114338574177835,-14.576745301996553)-- (5,-10.80384757729337);
					\draw [line width=1.5pt] (2,-16)-- (5,-10.80384757729337);
					\draw [line width=1.5pt] (6.33855808908444,-10.782548690559663)-- (1.114338574177835,-14.576745301996553);
					\draw [line width=1.5pt] (6.33855808908444,-10.782548690559663)-- (2,-16);
					\draw [line width=1.5pt] (-4.757500898992794,-17.209829319120182)-- (1.114338574177835,-14.576745301996553);
					\draw [line width=1.5pt] (-4.757500898992794,-17.209829319120182)-- (2,-16);
					\draw [line width=1.5pt] (-4.757500898992794,-17.209829319120182)-- (-7,-10.803847577293364);
					\draw [line width=1.5pt] (-4.757500898992794,-17.209829319120182)-- (-5.501742400013352,-10.756949170566905);
					\draw [line width=1.5pt] (-5.501742400013352,-10.756949170566905)-- (-4.775659706539203,-4.386720039274386);
					\draw [line width=1.5pt] (-4.775659706539203,-4.386720039274386)-- (-7,-10.803847577293364);
					\draw [line width=1.5pt] (-4.775659706539203,-4.386720039274386)-- (2,-5.607695154586736);
					\draw [line width=1.5pt] (-4.775659706539203,-4.386720039274386)-- (1.1143385741778362,-6.937153039137263);
					\draw [line width=1.5pt] (2,-5.607695154586736)-- (6.33855808908444,-10.782548690559663);
					\draw [line width=1.5pt] (6.33855808908444,-10.782548690559663)-- (1.1143385741778362,-6.937153039137263);
					\begin{scriptsize}
						\draw [fill=uuuuuu] (-4,-16) circle (5.5pt);
						\draw [fill=uuuuuu] (2,-16) circle (5.5pt);
						\draw [fill=uuuuuu] (5,-10.80384757729337) circle (5.5pt);
						\draw [fill=uuuuuu] (2,-5.607695154586736) circle (5.5pt);
						\draw [fill=uuuuuu] (-4,-5.607695154586734) circle (5.5pt);
						\draw [fill=uuuuuu] (-7,-10.803847577293364) circle (5.5pt);
						\draw [fill=uuuuuu] (1.1143385741778362,-6.937153039137263) circle (5.5pt);
						\draw [fill=uuuuuu] (-3.2963820752829553,-6.937153039137263) circle (5.5pt);
						\draw [fill=uuuuuu] (-5.501742400013352,-10.756949170566905) circle (5.5pt);
						\draw [fill=uuuuuu] (-3.2963820752829562,-14.576745301996551) circle (5.5pt);
						\draw [fill=uuuuuu] (1.114338574177835,-14.576745301996553) circle (5.5pt);
						\draw [fill=uuuuuu] (3.3196988989082326,-10.75694917056691) circle (5.5pt);
						\draw [fill=uuuuuu] (-4.757500898992794,-17.209829319120182) circle (5.5pt);
						\draw [fill=uuuuuu] (-4.775659706539203,-4.386720039274386) circle (5.5pt);
						\draw [fill=uuuuuu] (6.33855808908444,-10.782548690559663) circle (5.5pt);
					\end{scriptsize}
				\end{tikzpicture}
				
				\caption{\rm 
					Graph $C_6(2,3)$.
				}
				\label{001}
			\end{center}
}\end{figure}
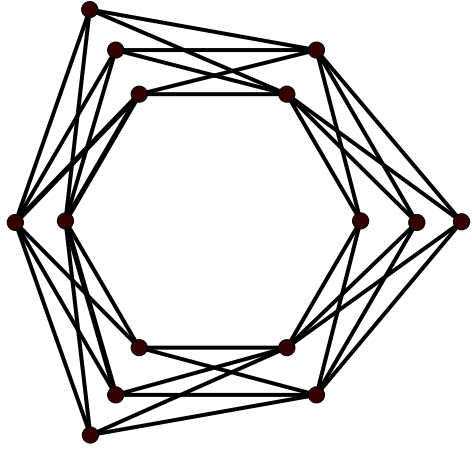}

The following proposition shows that graph $C_6(k,l)$ with $k\neq l$ is a counterexample to Conjecture~\ref{conjecture Sparl}.

\begin{proposition}
	{Let $k$ and $l$ be positive integers, and let the graph $C_6(k,l)$. 
	Then $C_6(k,l)$ is strongly-distance balanced if and only if $k=l$, while $d(u,W_{u,v})= d(v,W_{v,u})$ holds for
	every pair of adjacent vertices $u, v$ of $C_6(k,l)$}.
\end{proposition}
\begin{proof}
	Observe that $C_6(k,l)$ is regular if and only if $k=l$. It follows that for $k\neq l$, the graph $C_6(k,l)$ is not strongly-distance-balanced. Moreover, for $k=l$, the graph $C_6(k,l)$ is vertex-transitive, and since every vertex-transitive graph is strongly-distance-balanced it follows that $C_6(k,l)$ is SDB if and only if $k=l$.
	
	Let $u=(x_0,0)$ and $v=(x_1,0)$. Observe that 
	\begin{align*} 
		D^1_2(u, v) &=(\{x_1\}\times (\mathbb{Z}_l \setminus \{0\}))\cup (\{x_5\}\times \mathbb{Z}_l), \\ 
		D^2_1(u, v) &=\left(\{x_0\}\times (\mathbb{Z}_k \setminus \{0\})\right)\cup (\{x_2\}\times \mathbb{Z}_k),\\
		D^2_3(u, v) &= (\{x_4\}\times \mathbb{Z}_k),\\
		D^3_2(u, v) &= (\{x_3\}\times \mathbb{Z}_l).\\
	\end{align*}
	It follows that $d(u,W_{u,v})=|D^1_2(u, v)| + 2\cdot |D^2_3(u, v)| =(2l-1)+2k=2k+2l-1$. 
	Similarly we have $d(v,W_{v,u})=|D^2_1(u, v)|+2\cdot |D^3_2(u, v)|=(2k-1)+2l=2k+2l-1$.  We conclude that $d(u,W_{u,v})= d(v,W_{v,u})$. Since the graph $C_6(k,l)$ is edge-transitive, it follows that the same holds for any pair of adjacent vertices. This concludes the proof.
\end{proof}

\section{Distance-balanced property in semisymmetric graphs}
\label{sec:semisymmetric}

The main goal for this section is to answer a question by Kutnar  {et al.}\ from \cite{KMMM05}.

 Symmetry is perhaps one of those purely mathematical concepts that has found wide applications in several other branches of science and in many of these problems, symmetry conditions are naturally blended with certain metric properties of the underlying graphs. Kutnar  {et al.}\ explored a purely metric property of being (strongly) distance-balanced in the context of graphs enjoying certain special symmetry conditions. They showed that vertex-transitive graphs are not only distance-balanced,  they are also strongly distance-balanced (see \cite{KMMM05}). Furthermore, since being vertex-transitive is not a necessary condition for a graph to be distance-balanced, it was therefore natural for the authors to explore the property of being distance-balanced within the class of {\it semisymmetric graphs}; a class of objects which are as close to vertex-transitive graphs as one can possibly get, that is, regular edge-transitive graphs which are not vertex-transitive. The smallest semisymmetric graph has $20$ vertices and its discovery is due to Folkman \cite{Folkman}, the initiator of this topic of research. 
 
 A semisymmetric graph is necessarily bipartite, with the two sets of bipartition coinciding with the two orbits of the automorphism group. Consequently, semisymmetric graphs have no automorphisms which switch adjacent vertices, and therefore, may arguably be considered as good candidates for graphs which are not distance-balanced.
 Indeed, Kutnar  {et al.}\ proved there are infinitely many semisymmetric graphs which are not distance-balanced, but there are also infinitely many semisymmetric graphs which are distance-balanced. They also wondered the following question.
 
 \begin{question}\cite[Question 4.6]{KMMM05}
 	Is it true that a distance-balanced semisymmetric graph is also strongly distance-balanced?
 \end{question}
 
 We next answer this question negatively  by giving a construction of an infinite family of semisymmetric DB graphs which are not SDB. Before embarking on the corresponding construction, we make the following observations about the distance-balanced property in semisymmetric graphs using certain graph product.

 Let $G$ and $H$ denote graphs. The {\it{lexicographic product of $G$ and $H$}}, denoted by $G[H]$, is the graph with vertex set $V(G)\times V(H)$ where two vertices $(g_1, h_1)$ and $(g_2, h_2)$ are adjacent if and only if $g_1\sim g_2$, or $g_1=g_2$ and $h_1 \sim h_2$. It turns out that the lexicographic product $G[H]$ is connected if and only if $G$ is connected. 
 
 Necessary and sufficient conditions under which the lexicographic product give rises to a distance-balanced graph are given in \cite{JKR}.
 
 \begin{lemma}\cite[Theorem 4.2]{JKR}\label{semi1}
 	Let $G$ and $H$ be connected graphs. Then, the lexicographic product $G[H]$ is distance-balanced if
 	and only if $G$ is distance-balanced and $H$ is regular.
 \end{lemma}

Kutnar  {et al.}\ also investigated the strongly distance-balanced property of lexicographic graph products. 

 \begin{lemma}\cite[Theorem 3.4]{KMMM05}\label{semi2}
	Let $G$ and $H$ be graphs such that  $G[H]$ is connected. Then, the lexicographic product $G[H]$ is strongly distance-balanced if
	and only if $G$ is strongly distance-balanced and $H$ is regular.
\end{lemma}

For constructions of several infinite families of semisymmetric  distance-balanced graphs the following result will be useful:

\begin{lemma}\cite[Proposition 4.3]{KMMM05}\label{semi3}
	Let $\G$ be a semisymmetric graph. Then for every positive integer $n$, the lexicographic product $\G[nK_1]$ is semisymmetric, where $nK_1$ denotes the empty graph of $n$ vertices.  
\end{lemma}
 
 With these results in mind, we would like to point out the desired construction can be given provided we find at least one connected distance-balanced semisymmetric graph which is not strongly distance-balanced. Namely, let $\G$ be such a graph.  Then combining together Lemma \ref{semi1} and Lemma \ref{semi3}, we have that $\G[nK_1]$ is a distance-balanced semisymmetric graph for every positive integer $n$. Additionally, since $\G$ is a connected graph which is not SDB, it follows from Lemma \ref{semi2} that $\G[nK_1]$ is not SDB. For every positive integer $n$, we thus have the lexicographic product $\G[nK_1]$ is a DB semisymmetric graph which is not SDB. 
Kutnar  {et al.}\ checked the list of all semisymmetric connected cubic graphs of order up to 768 \cite{CMMP}, and there are exactly 11 distance-balanced graphs in this list, all of them are also strongly distance-balanced. They also checked the list of all connected semisymmetric tetravalent graphs of order up to 100 from the list of Poto\v cnik and Wilson, and there are 26 distance-balanced graphs in this list, all of which are also strongly distance-balanced. In the meantime, Poto\v cnik and Wilson extended their list of connected tetravalent edge-transitive graphs up to 512 vertices \cite{PW20}, and using this extended list we were able to find examples of semisymmetric graphs which are distance-balanced but not strongly distance-balanced.

\begin{example}
Graphs 
$C4[150, 9]$, 
$C4[240, 60]$,
$C4[240, 61]$,
$C4[240, 105]$,
$C4[240, 168]$, 
$C4[288, 145]$,
$C4[288, 171]$,
$C4[288, 246]$, 
$C4[312, 40]$,
$C4[336, 46]$,
$C4[336, 49]$,
$C4[336, 107]$,
$C4[336, 129]$,
$C4[336, 135]$,
$C4[336, 157]$,
$C4[336, 166]$,
$C4[360, 177]$,
$C4[384, 81]$,
$C4[384, 85]$,
$C4[384, 341]$,
$C4[384, 380]$,
$C4[384, 462]$,
$C4[384, 499]$,
$C4[400, 44]$,
$C4[432, 163]$,
$C4[432, 164]$,
$C4[432, 198]$,
$C4[432, 229]$,
$C4[432, 241]$,
$C4[432, 253]$,
$C4[432, 274]$,
$C4[432, 282]$,
$C4[480, 126]$,
$C4[480, 131]$,
$C4[480, 300]$,
$C4[480, 359]$,
$C4[480, 453]$,
$C4[480, 461]$,
$C4[480, 520]$,
$C4[480, 523]$,
$C4[486, 68]$,
$C4[486, 69]$,
$C4[486, 74]$,
$C4[504, 154]$,
$C4[504, 155]$
defined in \cite{PW20}
are connected semisymmetric graphs of valency 4 which are distance-balanced but not strongly distance-balanced. (The parameter $n$ in $C4[n,i]$ denotes the order of the corresponding graph). 
Using the distance-orbit chart given in \cite{PW20} (where the sizes of orbits of the stabilizer $\Aut(\G)_u$ of a vertex $u$ at {distances} $0,1, \ldots, d$ from $u$ are shown) one can easily check the distance-balanced and strongly distance-balanced properties of the graph under consideration (the orbit sizes are given for representatives of bipartition sets). For example, the distance-orbit chart of the graph $C4[150,9]$ is 

\begin{center}
  \begin{tabular}{ |c | c | c | c | c | c | c | c | c | c |}
    \hline
   Distance &  0	& 1& 	2&	3&	4&	5&	6&	7&	8 \\ \hline
   White vertex & 1	&4&	$2,4^2$ &	$2^2,4^4$ & $2,4^7$ &	$2^2,4^9$ &	$2^3,4^7$ &	$1,2,4^2$ &\\ \hline
  Black vertex &   1 &	4	&$2,4^2$&	$2^2,4^4$&	$2^2,4^7$&	$2^2,4^9$&	$2,4^7$&	$1,2,4^2$&	$2$ \\ \hline
  \end{tabular}
\end{center}
This means that there are 4 vertices at distance 1 from a white vertex, 10 vertices at distance two (one orbit of size 2 and two orbits of size 4), 20 vertices at distance 3 (two orbits of size 2 and 4 orbits of size 4), and so on.
By the result of Balakrishnan  {et al.}\  \cite{BCPSSS}, a graph is distance-balanced if and only if the sum of the distances from a given vertex to all other vertices is independent of the chosen vertex, which can easily be verified from the distance-orbit chart. Similarly, a graph is strongly distance-balanced if {and only if} the number of vertices at distance $i$ from a given vertex is independent of the chosen vertex, which can also easily be read from the distance-orbit chart.
\end{example}

 \begin{corollary}
 	There exist infinite families of distance-balanced semisymmetric graphs which are not strongly distance-balanced.
 	\end{corollary}

\section{Recognition of SDB and NDB graphs}\label{sec:recognition of SDB and NDB}

Let $\G$ be a graph with $n$ vertices and $m$ edges. In \cite{BCPSSS} it is proved that it can be verified in $O(mn)$ time if $\G$ is distance-balanced. We will now prove that the same result holds for strongly distance-balanced graphs { and nicely distance-balanced graphs.}

\begin{proposition}
	Let $\G$ be a connected graph with $n$ vertices and $m$ edges. It can be checked in $O(mn)$ time if $\G$ is strongly distance-balanced.
\end{proposition}
\begin{proof}
	By Proposition~\ref{prop:SDB characterization} it follows that $\G$ is strongly distance-balanced if and only if $|S_i(u)|$ does not depend on the choice of vertex $u$, for any $i\in \{1,\ldots,d\}$ where $d$ is the diameter of $\G$. Using BFS algorithm, the sizes of sets $|S_i(u)|$ can be determined in $O(m)$ time, for any fixed vertex $u$. Calculating these numbers for every vertex of $\G$ can then be done in $O(mn)$ time.
\end{proof}

{
\begin{proposition}
	Let $\G$ be a connected graph with $m$ edges. It can be checked in $O(mn)$ time if $\G$ is nicely distance-balanced.
\end{proposition}
\begin{proof}
	Using the BFS algorithm, computing the distance from each vertex to all other vertices can be done in $O(mn)$ time, and this information can be stored, for example in a distance matrix. 
For a fixed edge $uv$, iterating over each vertex $w$ and checking whether $d(u,w)$ is smaller, larger or equal than $d(v,w)$, we can
compute the sizes of $W_{u,v}$ and $W_{v,u}$, which can be done in $O(n)$ time (for a single edge). Calculating the values of $W_{u,v}$ and $W_{v,u}$ can then be done in $O(mn)$ time.
\end{proof}}

\bigskip 

{For a graph $\G$ and a vertex $v$, one can construct the sets $S_i(v)$ of all vertices in $\G$ which are at distance $i$ from $v$. By Proposition \ref{prop:SDB characterization}, we observe that $\G$ is SDB if and only if the sizes of the sets $S_i(v)$ {do not} depend on the choice of $v$. In \cite{BCPSSS}, Balakrishnan {et al.}\, showed that a graph is distance-balanced if and only if the sum of the distances from a given vertex to all other vertices is independent of the chosen vertex. Namely, $\G$ is DB if and only if $\sum_{i} i|S_i(v)|$ is constant. Therefore, we conclude the paper with the following question.}
\begin{problem}
	Does there exist a characterization of NDB graphs in terms of sets {the} $S_i(v)$? 
\end{problem}

\section{Acknowledgements}
We would like to thank to anonymous referees for carefully reading the manuscript and helpful comments that improved the quality of the paper.
This work is supported in part by the Slovenian Research Agency (research programs P1-0285, P1-0404, research projects N1-0140, N1-0159, N1-0208, N1-0102 ,  J1-1691, J1-1694, J1-1695,  J1-2451,  and Young Researchers Grant).

\end{document}